\documentclass[11pt,a4paper]{article}
\usepackage[utf8]{inputenc}
\usepackage[english]{babel}
\usepackage[T1]{fontenc}
\usepackage{amsmath}
\usepackage{amsfonts}
\usepackage{amssymb}
\usepackage{amsthm}
\usepackage{mathrsfs}
\usepackage{dsfont}
\usepackage[left=2cm,right=2cm,top=3cm,bottom=3cm]{geometry}
\usepackage{accents}
\usepackage{color}

\RequirePackage[colorlinks,citecolor=blue,urlcolor=blue]{hyperref}	

\newcommand{\ubar}[1]{\underaccent{\bar}{#1}}

\newtheorem{theorem}{Theorem}[section]
\newtheorem{lemma}[theorem]{Lemma}
\newtheorem{proposition}[theorem]{Proposition}

\theoremstyle{remark}
\newtheorem{remark}[theorem]{Remark}

\numberwithin{equation}{section}

\DeclareMathOperator{\dtv}{d_{TV}}
\DeclareMathOperator{\He}{He}
\newcommand{\lip}[1]{[#1]_{\textup{Lip}}}

\providecommand{\keywords}[1]
{
	\textbf{\textit{Keywords--}} #1
}

\providecommand{\MSC}[1]
{
	\textbf{\textit{MSC Classification--}} #1
}

\title{Total variation distance between two diffusions in small time with unbounded drift: application to the Euler-Maruyama scheme}
\author{Pierre Bras\footnote{Sorbonne Universit\'e, Laboratoire de Probabilit\'es, Statistique et Mod\'elisation, UMR 8001, case 158, 4 pl. Jussieu, F-75252 Paris Cedex 5, France. E-mail: \texttt{pierre.bras@sorbonne-universite.fr} and \texttt{gilles.pages@sorbonne-universite.fr}.} \footnote{Corresponding author}$\ $, Gilles Pag\`es\footnotemark[1] $\ $ and Fabien Panloup\footnote{LAREMA, Facult\'e des Sciences, 2 Boulevard Lavoisier, Université d'Angers, 49045 Angers, France. E-mail: fabien.panloup@univ-angers.fr}}
\date{}

\begin{document}

\maketitle

\begin{abstract}
We give bounds for the total variation distance between the solutions to two stochastic differential equations starting at the same point and with close coefficients, which applies in particular to the distance between an exact solution and its Euler-Maruyama scheme in small time. 
We show that for small $t$, the total variation distance is of order $t^{r/(2r+1)}$ if the noise coefficient $\sigma$ of the SDE is elliptic and $\mathcal{C}^{2r}_b$, $r\in \mathbb{N}$ and if the drift is $C^1$ with bounded derivatives, using multi-step Richardson-Romberg extrapolation. We do not require the drift to be bounded. Then we prove with a counterexample that we cannot achieve a bound better than $t^{1/2}$ in general.
\end{abstract}


\MSC{65C30, 60H35}

\section{Introduction}

The convergence properties of Euler-Maruyama schemes to approximate the solution of a Stochastic Differential Equation (SDE) have been extensively studied, in particular for $L^p$ distances. However, the literature seems to lack some results about the convergence in total variation in small time.
More specifically, in this paper we consider the two following SDEs in $\mathbb{R}^d$ starting at the same point:
\begin{align*}
& X_0^x = x \in \mathbb{R}^d, & dX^x_t = b_1(t, X^x_t)dt + \sigma_1(t, X^x_t)dW_t, \\
& Y_0^x = x, & dY^x_t = b_2(t, Y^x_t) dt + \sigma_2(t, Y^x_t)dW_t,
\end{align*}
where $W$ is a Brownian motion.
We generally assume that for $i =1,2$, $b_i$ is Lipschitz continuous and that $\sigma_i$ is elliptic, bounded and Lipschitz continuous, but we do not assume that $b_i$ is bounded. Our objective is to give bounds of the total variation distance between the law of $X^x_t$ and the law of $Y^x_t$, denoted $\dtv(X_t^x, Y^x_t)$, as $t \to 0$. In particular, we apply our results to the case where $Y^x = \bar{X}^x$ is the one-step Euler-Maruyama scheme associated to the SDE $X$, given by
$$ dY^x_t = b_1(0,x)dt + \sigma_1(0,x)dW_t .$$

Such bounds are well known for $L^p$ distances and their associated Wasserstein distances and are known to be of order $t$ as $t \to 0$. Yet the literature seems to lack results as it comes to $\dtv$. If $\sigma_1=\sigma_2$ is constant, then it is classical background that $\dtv(X^x_t, Y^x_t)$ is of order $t$, using a Girsanov change of measure (see for example \cite[Proposition 4.1]{pages2020}) but this strategy cannot be applied to non-constant $\sigma$.
The difficulty of the total variation distance in small time is the following: considering its representation formula and comparing it with the $L^1$-Wasserstein distance, if $x$ and $y \in \mathbb{R}^d$ are close to each other and if $f:\mathbb{R}^d\to \mathbb{R}$ is Lipschitz continuous, then we can bound $|f(x)-f(y)|$ by $\lip{f}|x-y|$; whereas if $f$ is simply measurable and bounded, then we cannot directly bound $|f(x)-f(y)|$ in terms of $|x-y|$. Moreover the regularizing properties of the semi-group cannot be used in small time for the total variation distance.

Results in the literature focus on the Euler-Maruyama scheme.
In \cite{bally1996} is proved the convergence for a fixed time horizon $T>0$ and as $N\to\infty$, where $N$ is the number of steps in the Euler-Maruyama scheme on a finite horizon. More precisely, if $\sigma_1$ is elliptic and if $b_1$ and $\sigma_1$ are $\mathcal{C}^\infty$ with bounded derivatives (but $b_1$ and $\sigma_1$ are not supposed bounded themselves), then (\cite[Theorem 3.1]{bally1996})
$$ \forall x \in \mathbb{R}^d, \ \dtv(X^x_T,\bar{X}^{x,N}_{T}) \le \frac{K(T)(1+|x|^Q)}{N T^q} ,$$
where $\bar{X}^{x,N}_{T}$ stands for the Euler scheme with $N$ steps, where $Q$ and $q$ are positive exponents and where $K$ is a non-decreasing function depending on $b_1$ and $\sigma_1$.
The common strategy of proof for such bounds is to use Malliavin calculus in order to perform an integration by parts and to use bounds on the derivatives of the density.
However, we cannot infer a bound as $T \to 0$ since we do not know whether $K(T)/T^q \to 0$ as $T \to 0$ in general. In \cite{gobet2008} are given bounds in small time and as $N \to \infty$. Assuming that $\sigma_1$ is uniformly elliptic and that $b_1$ and $\sigma_1$ are bounded with bounded derivatives up to order $3$, then (\cite[Theorem 2.3]{gobet2008})
$$ \forall t \in (0,T], \ \forall x,y \in \mathbb{R}^d, \ |p(t,x,y) - \bar{p}^N(t,x,y)| \le \frac{K(T)T}{Nt^{(d+1)/2}} e^{-C|x-y|^2/t} ,$$
where $p$ and $\bar{p}^N$ denote the transition densities of $X^x$ and $\bar{X}^{x,N}$ respectively and where $C$ is a positive constant depending on $d$ and on the bounds on $b_1$ and $\sigma_1$ and on their derivatives. However, we cannot directly use this result for the total variation distance: taking $N=1$ yields
$$ \dtv(X^x_t, \bar{X}^{x}_t) = \int_{\mathbb{R}^d} |p(t,x,y) - \bar{p}^N(t,x,y)| dy \le K(T)T t^{-1/2} \int_{\mathbb{R}^d} \frac{1}{t^{d/2}} e^{-C|x-y|^2/t} ,$$
giving a bound in $t^{-1/2}$ which does not converge to $0$ as $t \to 0$. Moreover, \cite{gobet2008} assumes that $b_1$ is bounded. \cite{bencheikh2020} focuses on the case where $b_1$ is bounded and measurable but not necessarily regular and where $\sigma_1$ is constant; it proves that the convergence in total variation of the Euler scheme on a finite horizon which is regularized with respect to the irregular drift $b_1$ and with step $h$, is of order $\sqrt{h}$.

In the present paper, we first prove a convergence rate of order $t^{1/3}$ for $\dtv(X^x_t, Y^x_t)$, provided that for $i=1,2$, $\sigma_i$ is elliptic, $\sigma_i$ and $b_i$ are Lipschitz-continuous with respect to their time variable and that $\sigma_i$ is $\mathcal{C}^2_b$ and $b_i$ is $\mathcal{C}^1$ and Lipschitz-continuous with respect to their space variable. More generally, if we furthermore assume that $\sigma$ is $\mathcal{C}^{2r}_b$, then we obtain a convergence rate of order $t^{r/(2r+1)}$. 
Letting $r \to \infty$, we also prove that if $\sigma \in \mathcal{C}^{\infty}_b$ with some technical condition on the derivatives of the densities of the random variables $X^x_t$ and $Y^x_t$, then the convergence rate is of order $t^{1/2}\exp(C\sqrt{\log(1/t)})$ which is "almost" $t^{1/2}$.
Moreover, we provide an example using the geometric Brownian motion where the convergence rate is exactly $t^{1/2}$, thus showing that we cannot achieve better bounds in general.
To prove the bound in $t^{r/(2r+1)}$, we use a multi-step Richardson-Romberg extrapolation \cite{richardson1911} \cite{lemaire2017}, which is a method imported from numerical analysis that we use in our case for theoretical purposes. It relies on a Taylor expansion with null coefficients up to some high order. Such method can be used in more general settings with regularization arguments in order to improve the convergence rates (in our case, we improve $t^{1/3}$ into $t^{r/(2r+1)}$).

Interestingly, the difference between the drift coefficients $b_1 - b_2$ does not need to be small for our bounds to be valid. This is because the dominant term in $\dtv(X^x_t, Y^x_t)$ comes from the the diffusion part.

Recent results (see \cite{clement2021}) establish a convergence in small time at rate $t^{1/2}$ for the Euler scheme of certain classes of diffusions driven by stable Lévy processes, not directly including the Brownian case. This approach relies on Malliavin calculus techniques. In this work the "standard" drift is replaced in the Euler scheme by the flow of the associated (noiseless) ODE. This seems to be specific to Lévy driven SDEs. Adapting this approach to our general continuous framework is not as straightforward as could be expected and would deserve further investigations for future work.

The total variation distance is closely related to the estimation of the density of the solution to an SDE and this density satisfies a Fokker-Planck Partial Differential Equation PDE \eqref{eq:backward_kolmogorov}. If the drift is bounded, then the density and its partial derivatives admit sub-gaussian Aronson's bounds (see \cite{friedman} and Section \ref{sec:friedman}).
However, giving estimates and bounds for the solution of the PDE in the case of unbounded drift appears to be more difficult, see \cite{lunardi1997}, \cite{cerrai2000}, \cite{bertoldi2005}.
Recent improvements have been made in \cite{menozzi2021} using the parametrix method.
Studying this case is useful to study the convergence in total variation of SDE's with unbounded drift, in particular for the Langevin equation, very popular in stochastic optimization, and which reads
$$ dX_t = -\nabla V(X_t) dt + \sigma(X_t)dW_t ,$$
where in many cases, $V : \mathbb{R}^d \to \mathbb{R}$ has quadratic growth and $\nabla V$ has linear growth (see for example \cite{bras2021-2}).

In order to deal with unbounded $b_i$, we propose two different methods. First, we use a localization argument and "cut" the drift $b_i$ into $\tilde{b}_i$ outside a compact set, so that we can use bounds from \cite{friedman} for the bounded drift case. We use the Girsanov formula to explicit the dependence of these bounds in $\|\tilde{b}_i\|_\infty$.  A second method consists in using the density estimates from \cite[Section 4]{menozzi2021} to improve the dependency with respect to the bounds in $x$. However this second approach relies on advanced parametrix methods which require further regularity assumptions on the coefficients of the SDE and which are not fully detailed for higher order derivatives. Our first method is clearly much more elementary, starting from a quite general bound established for any pair of integrable random vectors (see Theorem \ref{thm:dtv_bound_L1_density}) and calling upon a standard regularization strategy which combined with a multistep procedure, seems to be at least quasi-optimal in a very general framework.

\bigskip

\textsc{Notations}

We endow the space $\mathbb{R}^d$ with the canonical Euclidean norm denoted by $| \boldsymbol{\cdot} |$. For $x \in \mathbb{R}^d$ and for $R>0$, we denote $\textbf{B}(x,R) = \lbrace y \in \mathbb{R}^d : \ |y-x| \le R \rbrace$.

We denote $\mathcal{M}_d(\mathbb{R})$ the set of $d \times d$ matrices with real coefficients.

For $M \in (\mathbb{R}^d)^{\otimes k}$, we denote by $\|M\|$ its operator norm, i.e. $\|M\| = \sup_{u \in \mathbb{R}^{d\times k}, \ |u|=1} M \cdot u$. If $M : \mathbb{R}^d \to (\mathbb{R}^d)^{\otimes k}$, we denote $\|M\|_\infty = \sup_{x \in \mathbb{R}^d} \|M(x)\|$.

For $k \in \mathbb{N}$ and if $f:\mathbb{R}^d \to \mathbb{R}$ is $\mathcal{C}^k$, we denote by $\nabla ^k f : \mathbb{R}^d \rightarrow (\mathbb{R}^d)^{\otimes k}$ its differential of order $k$.
If $f$ is Lipschitz continuous, we denote by $[f]_{\text{Lip}}$ its Lipschitz constant.
If $f :(t,x) \in \mathbb{R} \times \mathbb{R}^d \mapsto f(t,x)$ is $\mathcal{C}^k$ with respect to $x$, we still denote by $\nabla^k f$ its differential with respect to $x$.



We denote the total variation distance between two distributions $\pi_1$ and $\pi_2$ on $\mathbb{R}^d$:
$$ \textstyle \dtv(\pi_1,\pi_2) = 2 \sup_{A \in \mathcal{B}(\mathbb{R}^d)} |\pi_1(A) - \pi_2(A)|.$$
Without ambiguity, if $Z_1$ and $Z_2$ are two $\mathbb{R}^d$-valued random vectors, we also write $\dtv(Z_1,Z_2)$ to denote the total variation distance between the law of $Z_1$ and the law of $Z_2$.
We have as well
$$ \dtv(\pi_1,\pi_2) = \sup \left\lbrace \int_{\mathbb{R}^d} fd\pi_1 - \int_{\mathbb{R}^d} fd\pi_1, \ f : \mathbb{R}^d \to [-1,1] \ \text{measurable} \right\rbrace .$$
Moreover, we recall that if $\pi_1$ and $\pi_2$ admit densities with respect to some measure $\lambda$, then
$$ \dtv(\pi_1,\pi_2) = \int_{\mathbb{R}^d} \left|\frac{d\pi_1}{d\lambda} - \frac{d\pi_2}{d\lambda}\right| d\lambda .$$

We denote by $\mathcal{W}_1$ the $L^1$-Wasserstein distance.

For $x \in \mathbb{R}^d$, we denote by $\delta_x$ the Dirac mass at $x$.

If $Z$ is a Markov process with values in $\mathbb{R}^d$, we denote, when it exists,
its transition probability from $x$ to $y \in \mathbb{R}^d$ between times $s<t$, $p_{_Z}(s,t,x,y)$.

In this paper, we use the notation $C$ and $c$ to denote positive constants, which may change from line to line.

\section{Main results}

We consider the two following SDEs in $\mathbb{R}^d$:
\begin{align}
\label{eq:SDE_def:1}
& X_0^x = x \in \mathbb{R}^d, & dX^x_t = b_1(t, X^x_t)dt + \sigma_1(t, X^x_t)dW_t, \ t \in [0,T], \\
\label{eq:SDE_def:2}
& Y_0^x = x, & dY^x_t = b_2(t, Y^x_t) dt + \sigma_2(t, Y^x_t)dW_t, \ t \in [0,T],
\end{align}
where $T$ is a finite time horizon, $b_i:\mathbb{R}^d \to \mathbb{R}^d$, $\sigma_i : \mathbb{R}^d \to \mathcal{M}_d(\mathbb{R})$, $i=1,2$, are Borel functions and $W$ is a standard $\mathbb{R}^d$-valued Brownian motion defined on a probability space $(\Omega, \mathcal{A}, \mathbb{P})$. The one-step Euler-Maruyama scheme of $X$, denoted $\bar{X}$, is defined by $\bar{X}^x=Y^x$ when $Y^x$ reads
\begin{equation}
\label{eq:euler_scheme_def}
dY^x_t = b_1(0,x)dt + \sigma_1(0,x)dW_t, \ t \in [0,T].
\end{equation}

To allievate notations, we also define
\begin{align}
\Delta b(x) := |b_1(0,x) - b_2(0,x)|, \quad \Delta\sigma(x) := |\sigma_1(0,x) - \sigma_2(0,x)|.
\end{align}
Let us remark that if $Y=\bar{X}$, then $\Delta b(x) = 0$ and $\Delta\sigma(x)=0$.
For $g : (t,x) \in [0,T] \times \mathbb{R}^d \mapsto g(t,x) \in \mathbb{R}^q$ and $r \in \mathbb{N}$, let us define the following assumptions:

$\bullet$ $\textup{Lip}_t(g)$: $g$ is Lipschitz continuous with respect to $t$, uniformly in $x$.

$\bullet$ $g \in \mathcal{C}^r$: $g$ is differentiable with respect to $x$ with continuous partial derivatives up to the order $r$.

$\bullet$ $g \in \mathcal{C}^r_b$: $g \in \mathcal{C}^r$ and is bounded with bounded partial derivatives up to the order $r$.

$\bullet$ $g \in \widetilde{\mathcal{C}}^r_b$: $g \in \mathcal{C}^r$ and has partial bounded derivatives up to the order $r$, but we do not assume that $g$ is bounded itself.

$\bullet$ For $\sigma : [0,T] \times \mathbb{R}^d \to \mathcal{M}_d(\mathbb{R})$, we say that $\sigma$ is (uniformly) elliptic if
\begin{equation}
\label{eq:elliptic_def}
\exists \ubar{\sigma}_0 >0, \ \forall x \in \mathbb{R}^d, \ \forall t \in [0,T], \ \sigma(t,x) \sigma(t,x)^\top \ge \ubar{\sigma}_0^2 I_d.
\end{equation}

\begin{theorem}
\label{thm:main:1}
Let $X$ and $Y$ be the solutions of the SDEs \eqref{eq:SDE_def:1} and \eqref{eq:SDE_def:2}. For $i=1,2$, assume $\textup{Lip}_t(b_i)$, $\textup{Lip}_t(\sigma_i)$, $\sigma_i \in \mathcal{C}^{2}_b$, $b_i \in \widetilde{\mathcal{C}}^1_b$ and $\sigma_i$ is elliptic. Then
\begin{equation}
\forall t \in [0,T], \ \forall x \in \mathbb{R}^d, \ \dtv(X^x_t, Y^x_t) \le C(t^{1/2} + \Delta\sigma(x))^{2/3} + Ce^{c|x|^2}t^{1/2},
\end{equation}
where the positive constants $C$ and $c$ only depend on $d$, $T$, $\ubar{\sigma}_0$, $\|\sigma_i\|_\infty$, $\lip{b_i}$, $\lip{\sigma_i}$ and $\|\nabla^2 \sigma_i\|_\infty$. In particular, if $Y=\bar{X}$ we have
$$ \dtv(X^x_t, Y^x_t) \le Ct^{1/3} + C e^{c|x|^2}t^{1/2}.$$
\end{theorem}

\begin{theorem}
\label{thm:main:2}
Let $X$ and $Y$ be the solutions of the SDEs \eqref{eq:SDE_def:1} and \eqref{eq:SDE_def:2}. For $i=1,2$, assume $\textup{Lip}_t(b_i)$, $\textup{Lip}_t(\sigma_i)$, that $\sigma_i \in \mathcal{C}^{2r}_b$, $b_i \in \widetilde{\mathcal{C}}^1_b$ and that $\sigma_i$ is elliptic. Then
\begin{equation}
\forall t \in [0,T], \ \forall x \in \mathbb{R}^d, \ \dtv(X^x_t, Y^x_t) \le C(t^{1/2} + \Delta\sigma(x))^{2r/(2r+1)} + C e^{c|x|^2}t^{1/2},
\end{equation}
where the positive constants $C$ and $c$ only depend on $d$, $T$, $\ubar{\sigma}_0$, on $\|\sigma_i\|_\infty$, on the bounds on the derivatives of $b_i$ and $\sigma_i$ and on their Lipschitz constants. In particular, if $Y=\bar{X}$ we have
$$ \dtv(X^x_t, Y^x_t) \le Ct^{r/(2r+1)} + C e^{c|x|^2}t^{1/2}.$$
\end{theorem}

\begin{remark}
In Theorems \ref{thm:main:1} and \ref{thm:main:2}, we can actually improve the dependency of the constants in $x$ in small time since we have more precisely:
$$ \forall t \in [0, Ce^{-2(2r+1)c|x|^2}], \quad \dtv(X^x_t, Y^x_t) \le C(t^{1/2} + \Delta\sigma(x))^{2r/(2r+1)} .$$
\end{remark}

\begin{remark}
We can adapt the framework of Theorem \ref{thm:main:2} to the study of SDEs with homogeneous vanishing noise, i.e. where for $a >0$,
$$ dX^x_t = b_1(t,X^x_t)dt + a \sigma_1(X^x_t) , \quad dY^x_t = b_2(t,Y^x_t) + a\sigma_2(Y^x_t) $$
and we identify the dependency in $a$ as $a \to 0$. Namely, with the same assumptions, for $a>0$ small enough we have
\begin{equation}
\dtv(X^x_t, Y^x_t) \le C e^{ca^{-1}|x|^2}t^{1/2} + Ca^{-(d+r)}(t^{1/2} + \Delta\sigma(x))^{2r/(2r+1)},
\end{equation}
where the constant $C$ does not depend on $a$. This bound is obtained adapting the proof of Theorem \ref{thm:main:2} in Section \ref{subsec:proof:2} and using that if $Z^x$ is the martingale $dZ^x_t = a \sigma(Z^x_t)dW_t$ then $(Z^x_t) \sim (\tilde{Z}^x_{a^2 t})$ where $\tilde{Z}^x_t = \sigma(\tilde{Z}^x_t)dW_t$ which does not depend on $a$.
\end{remark}

We can also improve the dependency in the initial condition $x \in \mathbb{R}^d$ using \cite{menozzi2021}, however at the expense of further regularity assumptions on $b_i$ and $\sigma_i$, $i=1,2$.
\begin{theorem}
\label{thm:main:2:menozzi}
Let $X$ and $Y$ be the solutions of the SDEs \eqref{eq:SDE_def:1} and \eqref{eq:SDE_def:2}. For $i=1,2$, assume $\textup{Lip}_t(b_i)$, $\textup{Lip}_t(\sigma_i)$, $\sigma_i \in \mathcal{C}^{2r+1}_b$, $b_i \in \widetilde{\mathcal{C}}^{2r}_b$ and  $\sigma_i$ is elliptic. Then
\begin{equation}
\forall t \in [0,T], \ \forall x \in \mathbb{R}^d, \ \dtv(X^x_t, Y^x_t) \le C\left(t^{1/2}(1+\Delta b(x)) + \Delta \sigma(x) + t(|b_1|+|b_2|)(0,x)\right)^{2r/(2r+1)},
\end{equation}
where the positive constants $C$ and $c$ only depend on $d$, $T$, $\ubar{\sigma}_0$, on $\|\sigma_i\|_\infty$, on the bounds on the derivatives of $b_i$ and $\sigma_i$ and on their Lipschitz constants. In particular, if $Y=\bar{X}$, we have
$$ \dtv(X^x_t, Y^x_t) \le Ct^{r/(2r+1)} \left(1+t^{1/2}(|b_1|+|b_2|)(0,x)|\right)^{2r/(2r+1)} .$$
\end{theorem}

\begin{remark}
Choosing $Y$ not to be the Euler-Maruyama scheme of $X$ but a general SDE and expressing the bounds in the Theorems in terms of $\Delta b(x)$ and $\Delta\sigma(x)$ allows to extend our results to more general couples of diffusions with "close" coefficients, for example to SDE solvers other than the genuine Euler-Maruyama scheme. Also, it is helpful to study perturbed SDEs, for example if we consider
\begin{equation}
dX^x_t = b_1(t,X^x_t)dt + a_1(t)\sigma(X^x_t)dW_t, \quad dY^x_t = b_2(t,Y^x_t)dt + a_2(t)\sigma(Y^x_t)dW_t
\end{equation}
where $|a_1(t) - a_2(t)| \to 0$ as $t \to \infty$. Then we have
$$ \dtv(X^x_{t+s}, Y^x_{t+s}) \le Ce^{c|x|^2}s^{1/2} + C(s^{1/2} + |a_1(t)-a_2(t)|)^{2r/(2r+1)} $$
and we obtain different convergence rates as $t \to \infty$ and $s \to 0$, depending on $(t,s)$. A noticeable example of this is the Langevin-simulated annealing SDE \cite{bras2021-2}.

Furthermore we remark that in Theorems \ref{thm:main:1} and \ref{thm:main:2}, the bounds do not depend on $\Delta b$, enhancing that the dominant term in the total variation comes from the diffusion part.
\end{remark}

To improve the rate of convergence from $t^{1/3}$ in Theorem \ref{thm:main:1} to $t^{r/(2r+1)}$ in Theorem \ref{thm:main:2}, we rely on a Richardson-Romberg extrapolation \cite{richardson1911} \cite{lemaire2017}; this argument can also be applied in a more general framework. The following proposition gives bounds on the total variation between two random vectors, knowing bounds on the $L^1$-Wasserstein distance and bounds on the partial derivatives of the densities up to some order.
\begin{theorem}
\label{thm:dtv_bound_L1_density}
Let $Z_1$ and $Z_2$ be two random vectors in $L^1(\mathbb{R}^d)$ and admitting densities $p_1$ and $p_2$ respectively with respect to the Lebesgue measure. Assume furthermore that $p_1$ and $p_2$ are $\mathcal{C}^{2r}$ with $r \in \mathbb{N}$ and that $\nabla^{k} p_i \in L^1(\mathbb{R}^d)$ for $i=1,2$ and $k=1,\ldots,2r$. Then we have
\begin{equation}
\dtv(Z_1,Z_2) \le C_{d,r} \mathcal{W}_1(Z_1,Z_2)^{2r/(2r+1)} \left( \int_{\mathbb{R}^d} \left( \| \nabla^{2r} p_1(\xi) \| + \| \nabla^{2r} p_2(\xi) \| \right) d\xi \right)^{1/(2r+1)}
\end{equation}
where the constant $C_{d,r}$ depends only on $d$ and on $r$.
\end{theorem}

If $\sigma \in \mathcal{C}^{\infty}_b$, then we also prove that we can "almost" get a convergence rate of order $t^{1/2}$.
\begin{theorem}
\label{thm:main:3}
Let $X$ and $Y$ be the solutions of the SDEs \eqref{eq:SDE_def:1} and \eqref{eq:SDE_def:2}. For $i=1,2$, assume $\textup{Lip}_t(b_i)$, $\textup{Lip}_t(\sigma_i)$, that $\sigma_i \in \mathcal{C}^{2r}_b$ for every $r \in \mathbb{N}$, $b_i \in \widetilde{\mathcal{C}}^1_b$, that $\sigma_i$ is elliptic and that $\Delta \sigma(x)=0$. Assume furthermore that if $Z$ and $V$ are the martingales $dZ_t = \sigma_1(t,Z_t)dW_t$ and $dV_t = \sigma_2(t,Z_t)dW_t$, then
\begin{align}
& \forall r \in \mathbb{N}, \ \forall t \in (0,T], \forall x,y \in \mathbb{R}^d, \ \| \nabla^{2r}_y p_{_Z}(0,t,x,y) \| + \| \nabla^{2r}_y p_{_V}(0,t,x,y) \| \le \frac{C_{2r}}{t^{(d+2r)/2}} e^{-c_{2r}|y-x|^2/t} \nonumber \\
\label{eq:assumption_C_tilde}
& \quad \text{ with } \limsup_{r \to \infty} \left(C_{2r} c_{2r}^{-d/2}\right)^{1/(2r)} < \infty.
\end{align}
(see Theorem \ref{thm:friedman}). Then
\begin{equation}
\forall t \in (0,T], \ \forall x \in \mathbb{R}^d, \ \dtv(X^x_t, Y^x_t) \le C e^{c|x|^2}t^{1/2} + Ct^{1/2} e^{c\sqrt{\log(1/t)}} ,
\end{equation}
where the positive constants $C$ and $c$ only depend on $d$, $T$, $\ubar{\sigma}_0$, on $\|\sigma_i\|_\infty$, on the bounds on the derivatives of $b_i$ and $\sigma_i$ and on their Lipschitz constants.
\end{theorem}

\begin{remark}
Assumption \eqref{eq:assumption_C_tilde} is satisfied in the case of a Brownian motion, which suggests that this assumption is satisfied in general provided that $\sigma$ is "regular enough". Indeed, if $dZ_t = \sigma dW_t$ with $\sigma \in \mathcal{M}_d(\mathbb{R})$ being non degenerate, then with $\Sigma := \sigma \sigma^\top$ we have for $t>0$ and $x,y \in \mathbb{R}^d$:
$$ p_{_Z}(0,t,x,y) = \frac{1}{\sqrt{\det(\Sigma)} t^{d/2}} \Phi\left(\Sigma^{-1/2}\frac{y-x}{\sqrt{t}}\right) , \quad \Phi(u) := \frac{1}{(2\pi)^{d/2}} e^{-|u|^2/2} .$$
Moreover for every $r \in \mathbb{N}$ and $u \in \mathbb{R}^d$ we have
$$ \left\| \frac{d^r}{du^r} \Phi(u) \right\| \le \frac{1}{(2\pi)^{d/2}} |\He_r(|u|)| e^{-|u|^2/2} $$
where $\He_r$ is the $r^{\text{th}}$ probabilist Hermite polynomial.
Following \cite{krasikov2004} we have
$$ \forall u \ge 0, \ |\He_{2r}(u)| e^{-u^2/2} \le C 2^{-r} \sqrt{r} \frac{(2r)!^2}{r!^2} \le C2^r \sqrt{r}, $$
using the Stirling formula for the last inequality. On the other hand, using \cite[22.14.15]{abramowitz1964} we have
$$ \forall u \ge 0, \ |\He_{2r}(u)|e^{-u^2/4} \le 2^{r+1} r! .$$
Then, for for every $\varepsilon \in (0,1]$,
\begin{align*}
|\He_{2r}(u)|e^{-u^2/2} & = \left|\He_{2r}(u) e^{-u^2/4}\right|^\varepsilon \left|\He_{2r}(u)e^{-u^2/2}\right|^{1-\varepsilon} e^{-\varepsilon u^2/4} \\
& \le C \left( 2^r r! \right)^\varepsilon \left(2^r r^{1/2}\right)^{1-\varepsilon} e^{-\varepsilon u^2/2}.
\end{align*}
Then if we choose $\varepsilon_r = \log^{-1}(r)$ for $r \ge 3$, we have
$ (r!)^{\varepsilon_r} \le e^{\varepsilon_r r\log(r)} = e^r $
so that
$$ \left\| \frac{d^r}{du^r} \Phi(u) \right\| \le C 2^{r\varepsilon_r} e^r 2^r r^{1/2} e^{-\varepsilon_r |u|^2/2} =: A_r e^{-\varepsilon_r |u|^2/2} $$
and then for $r \ge 3$ we have
\begin{align*}
\| \nabla^{2r}_y p_{_Z}(0,t,x,y) \| & \le \frac{\|\Sigma^{-1/2}\|^{2r}}{\sqrt{\det(\Sigma)} t^{(d+2r)/2}} \frac{d^{2r}}{du^{2r}} \Phi\left(\Sigma^{-1/2}\frac{y-x}{\sqrt{t}}\right) \\
& \le \frac{\|\Sigma^{-1/2}\|^{2r}}{\sqrt{\det(\Sigma)} t^{(d+2r)/2}} A_r e^{-\varepsilon_r \|\Sigma^{-1}\| |y-x|^2/(2t)}
\end{align*}
where $ \big(\left(\|\Sigma^{-1/2}\|^{2r} A_r \varepsilon_r^{-d/2} \right)^{1/(2r)}\big) $ is bounded. Thus Assumption \eqref{eq:assumption_C_tilde} is satisfied.
\end{remark}

\begin{remark}
For the Euler-Maruyama scheme \eqref{eq:euler_scheme_def}, with a slight abuse of notation, $x$ is used both for the starting point and in the definition of the drift and diffusion coefficients. The transition density should be considered for constant drift and diffusion coefficients in this case. However the results remain valid as the Euler scheme is simply a Brownian process.
\end{remark}


\section{Proof of the Theorems}

\subsection{Recalls on density estimates for SDEs with bounded drift}
\label{sec:friedman}

We recall results on the bounds for the density of the solution of the SDE using the theory of partial differential equations.
Let us consider a generic SDE:
\begin{equation}
\label{eq:generic_SDE}
Z_0^x = x \in \mathbb{R}^d , \quad dZ_t^x = b_{_Z} (t, Z_t^x) dt + \sigma_{_Z} (t, Z_t^x) dW_t, \ t \in [0,T].
\end{equation}
Then under regularity assumptions on $b_Z$ and on $\sigma_Z$, the transition probability $p_{_Z}$ exists and is solution of the backward Kolmogorov PDE:
\begin{align}
& p_{_Z}(t,t,x,\cdot) = \delta_x, \ t \in [0,T], \nonumber \\
\label{eq:backward_kolmogorov}
& \partial_s p_{_Z}(s,t,x,y) = \langle b_{_Z}(s,x) , \nabla_x p_{_Z}(s,t,x,y) \rangle + \frac{1}{2} \text{Tr}\left(\sigma_{_Z}^\top(s,x) \nabla^2_x p_{_Z}(s,t,x,y) \sigma_{_Z}(s,x)\right), \ s<t \in [0,T].
\end{align}
Moreover, $p_{_Z}$ and its derivatives satisfy sub-gaussian Aronson's bounds:
\begin{theorem}[\cite{friedman}, Chapter 9, Theorem 7]
\label{thm:friedman}
Let $Z$ be the solution of \eqref{eq:generic_SDE} and let $T>0$. Assume $\textup{Lip}_t(b_{_Z})$ and $\textup{Lip}_t(\sigma_{_Z})$, that $b_{_Z}, \sigma_{_Z} \in \mathcal{C}^r_b$ and that $\sigma_{_Z}$ is elliptic. Then for every $m_0=0,1$ and for every $0 \le m_1 + m_2 \le r$, $\nabla^{m_0+m_1}_x \nabla^{m_2}_y p_{_Z}$ exists and
\begin{equation}
\forall s<t \in [0,T], \ \forall x,y \in \mathbb{R}^d, \ \| \nabla^{m_0+m_1}_x \nabla^{m_2}_y p_{_Z}(s,t,x,y) \| \le \frac{C}{(t-s)^{(d+m_0+m_1+m_2)/2}} e^{-c|y-x|^2/(t-s)},
\end{equation}
where the constants $C$ and $c$ only depend on the bounds on $b_{_Z}$ and $\sigma_{_Z}$ and on their derivatives and their Lipschitz constants, on the modulus of ellipticity of $\sigma_{_Z}$, on $d$ and on $T$.
\end{theorem}

Let us also recall the recent result from \cite{menozzi2021} giving Aronson's bounds of the partial derivatives with respect to $y$ in the case where $b_Z$ is unbounded. Considering \cite[Section 4]{menozzi2021} with \cite[(3.1)]{menozzi2021}, we have the following result.
\begin{theorem}
\label{thm:menozzi}
Let $Z$ be the solution of \eqref{eq:generic_SDE} and let $T>0$. Assume $\textup{Lip}_t(b_{_Z})$ and $\textup{Lip}_t(\sigma_{_Z})$, that $b_{_Z} \in \widetilde{\mathcal{C}}^r_b$, $\sigma_{_Z} \in \mathcal{C}^{r+1}_b$ and that $\sigma_{_Z}$ is elliptic. Then for every $0 \le m \le r$, $ \nabla^{m}_y p_{_Z}$ exists and
\begin{equation}
\forall s<t \in [0,T], \ \forall x,y \in \mathbb{R}^d, \ \| \nabla^{m}_y p_{_Z}(s,t,x,y) \| \le \frac{C}{(t-s)^{(d+m)/2}} e^{-c|y-x|^2/(t-s)},
\end{equation}
where the constants $C$ and $c$ only depend on the bounds on $b_{_Z}$ and $\sigma_{_Z}$ and on their derivatives and on their Lipschitz constants, on the modulus of ellipticity of $\sigma_{_Z}$, on $d$ and on $T$.
\end{theorem}

\subsection{Preliminary results}

In order to apply the bounds on the densities from Theorem \ref{thm:friedman} to Theorem \ref{thm:dtv_bound_L1_density}, we first "cut" the drifts $b_1$ and $b_2$ on a compact set. That is, we instead consider the processes $\widetilde{X}$ and $\widetilde{Y}$ defined by
\begin{align}
\label{eq:def_Y:1}
& d\widetilde{X}^x_t = \tilde{b}^x_1(t, \widetilde{X}^x_t) dt + \sigma_1(t, \widetilde{X}_t^x) dW_t , \ t \in [0,T],\\
\label{eq:def_Y:2}
& d\widetilde{Y}^x_t = \tilde{b}^x_2(t, \widetilde{Y}^x_t) dt + \sigma_2(t, \widetilde{Y}_t^x) dW_t , \ t \in [0,T],
\end{align}
where $\tilde{b}_i$, $i=1,2$ is defined as follows. We choose $R > 0$ and we consider a $\mathcal{C}^\infty$ decreasing function $\psi : \mathbb{R}^+ \to \mathbb{R}^+$ such that $\psi = 1$ on $[0,R^2]$ and $\psi = 0$ on $[(R+1)^2,\infty)$ and we define $\tilde{b}^x_i(t, y) := b_i(t,y)\psi(|y-x|^2)$, so that $\tilde{b}_i^x$ is bounded:
\begin{equation}
\label{eq:b_tilde_bounded}
\textstyle \forall y \in \mathbb{R}^d, \ \forall t \in[0,T], \ |\tilde{b}_i^x(t,y)| \le \sup_{z \in \textbf{B}(x,R+1)} |b_i(t,z)| \le C(1+|x|) ,
\end{equation}
because $b_i$ is Lipschitz continuous.

\begin{lemma}
\label{lemma:dTV_X_Y}
Assume $\textup{Lip}_t(b_1)$, $\textup{Lip}_t(\sigma_1)$, $b_1 \in \widetilde{\mathcal{C}}^1_b$, $\sigma_1 \in \mathcal{C}^1_b$. Then for every $x \in \mathbb{R}^d$ and $t \in [0,T]$,
\begin{equation}
\dtv(X^x_t, \widetilde{X}^x_t) \le C(1+|b_1(0,x)|^2)t.
\end{equation}
\end{lemma}
\begin{proof}
Let $f : \mathbb{R}^d \to \mathbb{R}$ be measurable and bounded. We remark that on the event $\lbrace \sup_{s \in [0,t]} |X_s^x - x|^2 \le R^2 \rbrace$, we have $\widetilde{X}^x_t = X^x_t$, so that
$$ |\mathbb{E}f(\widetilde{X}^x_t) - \mathbb{E}f(X^x_t)| \le 2\|f\|_\infty \mathbb{P}\left(\sup_{s \in [0,t]} |X_s^x - x|^2 > R^2 \right) .$$
But using the inequality $|u+v|^2 \le 2|u|^2 + 2|v|^2$ we have
\begin{align*}
|X_t^x {-} x|^2 & \le 2 \left|\int_0^t b_1(s,X_s^x)ds \right|^2 {+} 2\left|\int_0^t \sigma_1(s,X_s^x)dW_s \right|^2 \le 2t \int_0^t |b_1(s,X_s^x)|^2 ds {+} 2\left|\int_0^t \sigma_1(s,X_s^x)dW_s \right|^2 \\
& \le 4t[b_1]_{\text{Lip}}^2 \left( \int_0^t |X_s^x-x|^2ds + \frac{1}{3}t^3 \right) + 4t^2|b_1(0,x)|^2 + 2\left|\int_0^t \sigma_1(s,X_s^x)dW_s \right|^2
\end{align*}
so that
\begin{align*}
\mathbb{E} \sup_{s \in [0,t]} |X_s^x-x|^2 & \le 4t\lip{b_1}^2 \int_0^t \left(\mathbb{E} \sup_{u \in [0,s]} |X_u^x-x|^2 \right) ds + \frac{4}{3}\lip{b_1}^2 t^4 + 4t^2|b_1(0,x)|^2 \\
& \quad + 2 \mathbb{E} \sup_{s \in [0,t]} \left|\int_0^s \sigma_1(u,X_u^x)dW_u \right|^2.
\end{align*}
Moreover using Doob's martingale inequality we have
\begin{align*}
\mathbb{E} \sup_{s \in [0,t]} \left|\int_0^s \sigma_1(u,X_u^x)dW_u \right|^2 \le 4 \mathbb{E}\left|\int_0^t \sigma_1(u,X^x_u) dW_u \right|^2 = 4 \mathbb{E} \int_0^t \sigma_1^2(u,X_u^x)du \le 4 \|\sigma_1\|_\infty^2 t.
\end{align*}
Then we define the non-decreasing deterministic process $ S_t := \mathbb{E} \sup_{s \in [0,t]} |X_s^x - x|^2 $
and we get the differential inequality (using $t \le T$)
$$ S_t \le 4t\left(T|b_1(0,x)|^2 + \frac{1}{3}\lip{b_1}^2 T^3 + 2\|\sigma_1\|_\infty^2 \right) + 4t[b_1]^2_{\text{Lip}} \int_0^t S_s ds ,$$
so the Gronwall lemma yields
$$ S_t \le 4t \left( T|b_1(0,x)|^2 + \frac{1}{3}\lip{b_1}^2 T^3 + 2\|\sigma_1\|_\infty^2 \right) e^{2t^2 \lip{b_1}^2} \le C(1+|b_1(0,x)|^2)t. $$
Using Markov's inequality, we have then
$$|\mathbb{E}f(\widetilde{X}^x_t) - \mathbb{E}f(X^x_t)| \le 2\|f\|_\infty \mathbb{P}\left(\sup_{s \in [0,t]} |X_s^x - x|^2 > R^2 \right) \le 2\|f\|_\infty \frac{C(1+|b_1(0,x)|^2)t}{R^2} .$$
\end{proof}

We can now apply Theorem \ref{thm:friedman} to $\widetilde{X}$ and to $\widetilde{Y}$ however the constants arising depend on the bound on $\|\tilde{b}_i^x\|_\infty$ and thus on $x$. In order to deal with the dependency in $\|\tilde{b}_i^x\|_\infty$, we apply the Girsanov formula and reduce to the null drift case.
\begin{proposition}
\label{prop:qian}
Let $Z^x$ be the solution of
\begin{equation}
\label{eq:def_Z}
Z_0^x = x, \quad dZ^x_t = \sigma_1(t,Z^x_t) dW_t, \ t \in [0,T] .
\end{equation}
Assume $\textup{Lip}_t(b_1)$, $\textup{Lip}_t(\sigma_1)$, $b_1 \in \widetilde{C}^1_b$, $\sigma_1 \in \mathcal{C}^1_b$ and $\sigma_1$ is elliptic. Then we have for every $t \in [0,T]$, $x,y \in \mathbb{R}^d$,
\begin{equation}
\label{eq:girsanov_p}
p_{_{\widetilde{X}}}(0,t,x,y) = p_{_Z}(0,t,x,y) + \int_0^t \mathbb{E}\left[U_s^x \langle \tilde{b}_1^x(s,Z_s^x), \nabla_x p_{_Z}(s,t,Z_s^x,y) \rangle \right] ds,
\end{equation}
where $\widetilde{X}$ is defined in \eqref{eq:def_Y:1} and $U^x$ is defined as
\begin{align}
& U_s^x = \exp\left( \int_0^s \langle g(u,Z_u^x) \tilde{b}_1^x(u,Z_u^x), dZ_u^x \rangle - \frac{1}{2} \int_0^s \langle g(u,Z_u^x)\tilde{b}_1^x(u,Z_u^x), \tilde{b}_1^x(u,Z_u^x) \rangle du \right), \\
& g = (\sigma_1 \sigma_1^\top)^{-1}.
\end{align}
\end{proposition}
\begin{proof}
First, note that since $\sigma_1$ is elliptic and since $\tilde{b}_1^x, \sigma_1 \in \mathcal{C}^1_b$, then $p_{_{\widetilde{X}}}$ and $p_{_Z}$ exist as well as $\nabla_x p_{_Z}$ (Theorem \ref{thm:friedman}).
We then use \cite[Theorem 2.4]{qian2003} extended to non-homogeneous diffusion processes. Following \cite[Remark 2.5]{qian2003}, since $\sigma_1$ is elliptic and bounded, the assumptions of \cite[Theorem 2.4]{qian2003} hold.
\end{proof}

We also have the following bounds on the process $U^x$.
\begin{lemma}
With the same assumptions as in Proposition \ref{prop:qian}, for every $p \ge 2$, $x \in \mathbb{R}^d$ and $t \in [0,T]$ we have
\begin{equation}
\mathbb{E} \left[ \sup_{s \in [0,t]} |U_s^x|^p \right] \le e^{Cp^2\|\tilde{b}_1^x\|_\infty^2 t} .
\end{equation}
\end{lemma}
\begin{proof}
We recall that for every $q \ge 1$, the process $(U^x)^q$ is a martingale with
$$ d(U_s^x)^q = q(U_s^x)^q \big\langle g(s,Z_s^x) \tilde{b}_1^x(s,Z_s^x), \sigma_1(s,Z_s^x) dW_s \big\rangle .$$
Thus, Doob's martingale inequality yields
$$ \mathbb{E} \left[ \sup_{s \in [0,t]} |U_s^x|^{2q} \right] \le Cq^2 \|\tilde{b}_1^x\|_\infty^2 \mathbb{E} \int_0^t |U_s^x|^{2q} ds \le Cq^2 \|\tilde{b}_1^x\|_\infty^2 \int_0^t \mathbb{E} \left[ \sup_{u \in [0,s]} |U_s^x|^{2q} \right] ds. $$
So with $U^x_0=1$ we obtain
$$ \mathbb{E} \left[ \sup_{s \in [0,t]} |U_s^x|^{2q} \right] \le e^{Cq^2\|\tilde{b}_1^x\|_\infty^2 t} .$$
\end{proof}


\begin{lemma}
With the same assumptions as in Proposition \ref{prop:qian}, we have for every $x,y \in \mathbb{R}^d$ and $t \in [0,T]$,
\begin{equation}
\label{eq:girsanov_b}
\left| \int_0^t \mathbb{E}\left[U_s^x \langle \tilde{b}_1^x (s,Z_s^x), \nabla_x p_{_Z}(s,t,Z_s^x,y) \rangle \right] ds \right| \le C e^{C\|\tilde{b}_1^x\|_\infty^2} \frac{e^{-c|y-x|^2/t}}{t^{(d-1)/2}}.
\end{equation}
\end{lemma}
\begin{proof}
We use Theorem \ref{thm:friedman} on the process $Z^x$, which yields bounds with constants depending on $\sigma_1$ but not on $\tilde{b}_1^x$. We obtain for every $q \ge 1$ and for every $s \in [0,t]$:
\begin{align*}
\mathbb{E}\left| \nabla_x p_{_Z}(s,t,Z_s^x,y) \right|^q & = \int_{\mathbb{R}^d} | \nabla_x p_{_Z}(s,t,\xi,y)|^q p_{_Z}(0,s,x,\xi) d\xi \\
& \le \frac{C_q}{(t-s)^{(q+(q-1)d)/2}} \int_{\mathbb{R}^d} \frac{1}{(s(t-s))^{d/2}} \exp\left(-c_q \left(\frac{|y-\xi|^2}{t-s} + \frac{|\xi-x|^2}{s} \right)\right) d\xi \\
& \le \frac{C_q}{t^{d/2}} \frac{1}{(t-s)^{(q+(q-1)d)/2}} e^{-c_q|y-x|^2/t},
\end{align*}
where we used Lemma \ref{lemma:friedman_Lemma_7} in the appendix.
Then for $p^{-1} + q^{-1} = 1$ and $p\ge 2$, using the Hölder inequality we have
\begin{align*}
& \left|\int_0^t \mathbb{E}\left[U_s^x \langle \tilde{b}_1^x(s, Z_s^x), \nabla_x p_{_Z}(s,t,Z_s^x,y) \rangle \right] ds \right| \\
& \quad \quad \quad \quad \quad \quad \quad \le \|\tilde{b}_1^x\|_\infty \Big(\sup_{s \in [0,t]} \mathbb{E}|U_s^x|^p \Big)^{1/p} \int_0^t \left(\mathbb{E}\left| \nabla_x p_{_Z}(s,t,Z_s^x,y) \right|^q \right)^{1/q} ds \\
& \quad \quad \quad \quad \quad \quad \quad \le \|\tilde{b}_1^x\|_\infty e^{Cp \|\tilde{b}_1^x\|_\infty^2 t} \frac{C_q e^{-c_q|y-x|^2/t}}{t^{d/(2q)}} \int_0^t \frac{ds}{(t-s)^{(1+(1-q^{-1})d)/2}} .
\end{align*}
The integral in $ds$ converges under the condition $q < d/(d-1)$ if $d>1$, and for any value of $q>1$ if $d=1$. Then performing the change of variable $s=tu$ we obtain
\begin{align*}
\left| \int_0^t \mathbb{E}\left[U_s^x \langle \tilde{b}_1^x(s, Z_s^x), \nabla_x p_{_Z}(s,t,Z_s^x,y) \rangle \right] ds \right| \le \|\tilde{b}_1^x\|_\infty e^{Cp \|\tilde{b}_1^x\|_\infty^2 T} \frac{C_q e^{-c_q|y-x|^2/t}}{t^{(d-1)/2}} \le C e^{C\|\tilde{b}_1^x\|_\infty^2} \frac{e^{-c|y-x|^2/t}}{t^{(d-1)/2}}.
\end{align*}

\end{proof}

\subsection{Proof of Theorem \ref{thm:main:1}}
\label{sec:thm1_proof}

\begin{lemma}
\label{lemma:cut_drift}
We have for every $x \in \mathbb{R}^d$ and $t \in [0,T]$:
\begin{align*}
\dtv(X^x_t, Y^x_t) \le \dtv(Z^x_t, V^x_t) + Ce^{C|x|^2}t^{1/2},
\end{align*}
where $dZ^x_t = \sigma_1(t,Z^x_t)dW_t$ and $dV^x_t = \sigma_2(t,V^x_t)dW_t$.
\end{lemma}
\begin{proof}
Let us write
\begin{align*}
\dtv(X^x_t, Y^x_t) \le \dtv(X^x_t, \widetilde{X}^x_t) + \dtv(\widetilde{X}^x_t,Z^x_t) + \dtv(Z^x_t,V^x_t) + \dtv(V^x_t, \widetilde{Y}^x_t) + \dtv(\widetilde{Y}_t^x, Y_t^x),
\end{align*}
where $\widetilde{X}$ and $\widetilde{Y}$ are defined in \eqref{eq:def_Y:1} and \eqref{eq:def_Y:2}.
Using Lemma \ref{lemma:dTV_X_Y} with \eqref{eq:b_tilde_bounded}, we have
$$ \dtv(X^x_t, \widetilde{X}^x_t) + \dtv(\widetilde{Y}_t^x, Y_t^x)  \le C(1+|x|^2)t .$$
Using the formula \eqref{eq:girsanov_p} and the inequality \eqref{eq:girsanov_b}, we have
\begin{align*}
\dtv(\widetilde{X}^x_t, Z^x_t) & {=} \int_{\mathbb{R}^d} |p_{_{\widetilde{X}}}(0,t,x,y) {-} p_{_Z}(0,t,x,y)| dy = \int_{\mathbb{R}^d} \left|\int_0^t \mathbb{E}\left[U_s^x \langle \tilde{b}_1^x(s,Z_s^x), \nabla_x p_{_Z}(s,t,Z_s^x,y) \rangle \right] ds \right| dy \\
& \le Ce^{C\|\tilde{b}_1^x\|_\infty^2} t^{1/2} \int_{\mathbb{R}^d} \frac{e^{-c|x-y|^2/t}}{t^{d/2}} dy \le Ce^{C|x|^2} t^{1/2},
\end{align*}
where we used \eqref{eq:b_tilde_bounded}. The term $\dtv(\widetilde{Y}^x_t, V^x_t)$ is treated likewise.
\end{proof}

We now prove Theorem \ref{thm:main:1}.
\begin{proof}
Let us introduce an artificial regularization. For $\varepsilon>0$ and using Lemma \ref{lemma:cut_drift} we have
\begin{align}
\label{eq:dTV_D_6}
\dtv(X_t^x, Y_t^x) & \le Ce^{C|x|^2}t^{1/2} + \dtv(Z^x_t, Z^x_t + \sqrt{\varepsilon} \zeta) + \dtv(Z^x_t + \sqrt{\varepsilon} \zeta, V^x_t + \sqrt{\varepsilon}\zeta) + \dtv(V^x_t+ \sqrt{\varepsilon}\zeta, V^x_t)
\end{align}
where $\zeta \sim \mathcal{N}(0,I_d)$ and is independent of the Brownian motion $W$.

\smallskip

$\bullet$ Let $f:\mathbb{R}^d \to \mathbb{R}$ be measurable and bounded and let us define
\begin{equation}
\label{eq:varphi_def}
\varphi : y \in \mathbb{R}^d \mapsto \mathbb{E} f(Z^x_t + y) = \int_{\mathbb{R}^d} f(\xi + y) p_{_Z}(0,t,x,\xi)d\xi = \int_{\mathbb{R}^d} f(\xi)p_{_Z}(0,t,x,\xi-y)d\xi .
\end{equation}
Then $\varphi$ is $\mathcal{C}^2$ with
$$ \nabla^2 \varphi(y) = \int_{\mathbb{R}^d} f(\xi) \nabla^2_y p_{_Z}(0,t,x,\xi-y) d\xi .$$
Moreover, using Theorem \ref{thm:friedman}, we have
$$ \| \nabla^2_y p_{_Z}(0,t,x,\xi-y) \| \le \frac{C}{t^{(d+2)/2}} e^{-c|x-\xi+y|^2/t} ,$$
where the constants $C$ and $c$ do not depend on $\tilde{b}_1^x$. This implies that for every $y \in \mathbb{R}^d$,
$$ \|\nabla^2 \varphi(y)\|_\infty \le C\|f\|_\infty t^{-1} \int_{\mathbb{R}^d} \frac{1}{t^{d/2}} e^{-c|x-\xi+y|^2/t}d\xi \le C\|f\|_\infty t^{-1} .$$
Then using the Taylor formula, for every $y \in \mathbb{R}^d$ there exists $\tilde{y} \in (0,y)$ such that
$$ \varphi(y) = \varphi(0) + \nabla \varphi(0) \cdot y + \frac{1}{2} \nabla^2 \varphi(\tilde{y}) \cdot y^{\otimes 2} $$
and then for some random $\tilde{\zeta} \in (0,\zeta)$ we have
\begin{align*}
|\mathbb{E}f(Z^x_t + \sqrt{\varepsilon}\zeta) - \mathbb{E}f(Z^x_t)| & = |\mathbb{E} \varphi(\sqrt{\varepsilon}\zeta) - \varphi(0)| = \left| \sqrt{\varepsilon} \mathbb{E}[\nabla \varphi(0) \cdot \zeta] + \frac{\varepsilon}{2} \mathbb{E}[ \nabla^2 \varphi(\sqrt{\varepsilon}\tilde{\zeta}) \cdot \zeta^{\otimes 2}] \right| \\
& \le C\varepsilon\|f\|_\infty t^{-1},
\end{align*}
where we used that $\mathbb{E}[\nabla\varphi(0) \cdot \zeta] = \nabla \varphi(0) \cdot \mathbb{E}[\zeta] = 0$. This way we obtain
$$ \dtv(Z^x_t, Z^x_t + \sqrt{\epsilon}\zeta) \le C\varepsilon t^{-1} .$$
The term $\dtv(V^x_t, V^x_t + \sqrt{\epsilon}\zeta)$ is treated likewise.

\smallskip

$\bullet$ Let $f:\mathbb{R}^d \to \mathbb{R}$ be measurable and bounded and let us define
\begin{equation}
\label{eq:f_epsilon_def}
f_\varepsilon : y \mapsto \mathbb{E}f(y + \sqrt{\varepsilon}\zeta) = \frac{1}{(2\pi)^{d/2}} \int_{\mathbb{R}^d} f(y+\sqrt{\varepsilon}\xi) e^{-|\xi|^2/2} d\xi = \frac{1}{(2\pi \varepsilon)^{d/2}} \int_{\mathbb{R}^d} f(\xi) e^{-|\xi-y|^2/(2\varepsilon)} d\xi.
\end{equation}
Then $f_\varepsilon$ is $\mathcal{C}^1$ with
\begin{align*}
\nabla f_\varepsilon(y) & = \frac{1}{(2\pi\varepsilon)^{d/2}} \int_{\mathbb{R}^d} f(\xi) \frac{\xi-y}{\varepsilon} e^{-|\xi-y|^2/(2\varepsilon)} d\xi = \frac{\varepsilon^{-1/2}}{(2\pi)^{d/2}} \int_{\mathbb{R}^d} f(y+\sqrt{\varepsilon}\xi) \xi e^{-|\xi|^2/2} d\xi \\
& = \varepsilon^{-1/2} \mathbb{E}[f(y+\sqrt{\varepsilon}\zeta) \zeta]
\end{align*}
and then
\begin{equation}
\label{eq:f_epsilon_lip}
\lip{f_\varepsilon} \le \|f\|_\infty \varepsilon^{-1/2} \mathbb{E}|\zeta| \le C\|f\|_\infty \varepsilon^{-1/2} .
\end{equation}
So that
\begin{align}
& |\mathbb{E}f(Z^x_t+\sqrt{\varepsilon}\zeta) - \mathbb{E}f(V^x_t+\sqrt{\varepsilon}\zeta)| = |\mathbb{E}f_\varepsilon(Z^x_t) - \mathbb{E}f_\varepsilon(V^x_t)| \nonumber \\
\label{eq:D_3}
& \quad \le \frac{C\|f\|_\infty}{\sqrt{\varepsilon}} \|Z^x_t - V^x_t\|_1 \le \frac{C\|f\|_\infty}{\sqrt{\varepsilon}} (t + t^{1/2}\Delta\sigma(x)) ,
\end{align}
where we used Lemma \ref{lemma:strong_error} in the Appendix.
This implies that
$$ \dtv(Z^x_t+\sqrt{\varepsilon}\zeta, V^x_t+\sqrt{\varepsilon}\zeta) \le C\varepsilon^{-1/2}(t + t^{1/2}\Delta\sigma(x)) .$$

\smallskip

$\bullet$ \textbf{Conclusion :} Considering \eqref{eq:dTV_D_6}, we have
\begin{align*}
\dtv(X^x_t,Y^x_t) \le Ce^{C|x|^2}t^{1/2} + C\varepsilon t^{-1} + C\varepsilon^{-1/2}(t + t^{1/2}\Delta\sigma(x)).
\end{align*}
We now choose $\varepsilon = \big[t(t+t^{1/2}\Delta\sigma(x))\big]^{2/3}$, so that
$$ \dtv(X^x_t, Y^x_t) \le Ce^{C|x|^2}t^{1/2} + C(t^{1/2} + \Delta\sigma(x))^{2/3} .$$

\end{proof}

\subsection{Proof of Theorem \ref{thm:main:2} using Theorem \ref{thm:dtv_bound_L1_density}}
\label{subsec:proof:2}

We first prove Theorem \ref{thm:dtv_bound_L1_density}.
\begin{proof}
Let $f:\mathbb{R}^d\to\mathbb{R}$ be measurable and bounded, let $\varepsilon>0$ and let $\zeta \sim \mathcal{N}(0,I_d)$ be independent of $(Z_1,Z_2)$. We have
\begin{align}
|\mathbb{E}f(Z_1) - \mathbb{E}f(Z_2)| & \le \left|\mathbb{E}f(Z_1) - \sum_{i=1}^r w_i \mathbb{E}f_{\varepsilon/n_i}(Z_1) \right| + \left|\sum_{i=1}^r w_i \mathbb{E}f_{\varepsilon/n_i}(Z_1) - \sum_{i=1}^r w_i \mathbb{E}f_{\varepsilon/n_i}(Z_2) \right| \nonumber \\
\label{eq:TI_Z1_Z2}
& \quad + \left|\sum_{i=1}^r w_i \mathbb{E}f_{\varepsilon/n_i}(Z_2) - \mathbb{E}f(Z_2)\right|,
\end{align}
where $f_\varepsilon$ is defined as in \eqref{eq:f_epsilon_def} and where the $n_i$'s and the $w_i$'s will be defined later.

Let $\varphi$ be as defined in \eqref{eq:varphi_def} replacing $Z^x_t$ by $Z_1$.
Then, $\varphi$ is differentiable up to the order $2r$ and for all $k=0,1,\ldots,2r$:
$$ \nabla^k \varphi(y) = (-1)^k \int_{\mathbb{R}^d} f(\xi) \nabla^k p_{1}(\xi-y)d\xi .$$
Using the Taylor formula up to order $2r$, for every $y \in \mathbb{R}^d$ there exits $\tilde{y} \in (0,y)$ such that
$$ \varphi(y) = \varphi(0) + \sum_{k=1}^{2r-1} \frac{\nabla^k \varphi(0)}{k!} \cdot y^{\otimes k} + \frac{\nabla^{2r} \varphi(\tilde{y})}{(2r)!} \cdot y^{\otimes 2r} .$$
Moreover, we have
\begin{equation}
\label{eq:ML_inequality}
\left|\nabla^{2r} \varphi(\tilde{y}) \cdot y^{\otimes 2r} \right| \le C\|f\|_\infty  |y|^{2r} \int_{\mathbb{R}^d} \| \nabla^{2r}p_1(\xi) \| d\xi .
\end{equation}
Then there exists a random $\tilde{\zeta} \in (0,\zeta)$ such that
\begin{align}
& \mathbb{E}f(Z_1+\sqrt{\varepsilon}\zeta)-\mathbb{E}f(Z_1) = \mathbb{E}\varphi(\sqrt{\varepsilon}\zeta ) - \varphi(0) = \sum_{k=1}^{2r-1} \frac{\nabla^k \varphi(0)}{k!} \varepsilon^{k/2} \cdot \mathbb{E} [\zeta^{\otimes k}] + \frac{\mathbb{E}[\nabla^{2r} \varphi(\sqrt{\varepsilon}\tilde{\zeta}) \cdot \zeta^{\otimes 2r}] }{(2r)!} \varepsilon^r \nonumber \\
\label{eq:D_2_bound:2}
& \quad = \sum_{k=1}^{r-1} \frac{\nabla^{2k} \varphi(0)}{(2k)!} \varepsilon^{k} \cdot \mathbb{E} [\zeta^{\otimes 2k}] + \frac{\mathbb{E}[\nabla^{2r} \varphi(\sqrt{\varepsilon}\tilde{\zeta}) \cdot \zeta^{\otimes 2r}] }{(2r)!} \varepsilon^r =: \sum_{k=1}^{r-1} \beta_k(t)\varepsilon^k + \tilde{\beta}_r(t,\varepsilon)\varepsilon^r,
\end{align}
because if $k$ is odd, then $\mathbb{E}[\zeta^{\otimes k}] = 0$.
We now rely on a multi-step Richardson-Romberg extrapolation \cite[Appendix A]{lemaire2017}. Let us denote the refiners $n_i=2^{i-1}$ and the auxiliary sequences and weights
\begin{equation}
u_k := \left(\prod_{\ell=1}^{k-1} (1-2^{-\ell})\right)^{-1}, \ v_k := (-1)^k 2^{-k(k+1)/2} u_{k+1}, \ w_k := u_k v_{r-k}, \ k=1,\ldots,r .
\end{equation}
These weights are the unique solution to the $r \times r$ Vandermonde system
\begin{equation}
\label{eq:vandermonde}
\sum_{i=1}^r w_i n_i^{-k} = \left\lbrace \begin{array}{ll} 1 & \text{ if } k=0, \\ 0 & \text{else.} \end{array} \right. , \quad k=0,1,\ldots,r-1 .
\end{equation}
Then we have
\begin{align}
\sum_{i=1}^r w_i \left(\mathbb{E}f(Z_1+\sqrt{\varepsilon/n_i}\zeta) - \mathbb{E}f(Z_1)\right) & = \sum_{i=1}^r w_i \sum_{k=1}^{r-1} \beta_k(t) \varepsilon^k n_i^{-k} + \sum_{i=1}^r w_i \tilde{\beta}_r(t,\varepsilon/n_i)\varepsilon^r n_i^{-r} \nonumber \\
& = \sum_{k=1}^{r-1} \varepsilon^k \beta_k(t) \sum_{i=1}^r w_in_i^{-k} + \varepsilon^r \sum_{i=1}^r \tilde{\beta}_r(t,\varepsilon/n_i)w_i n_i^{-r} \nonumber \\
\label{eq:ML_expansion}
& = \varepsilon^r \sum_{i=1}^r \tilde{\beta}_r(t,\varepsilon/n_i)w_i n_i^{-r},
\end{align}
where we used \eqref{eq:vandermonde} in the last equation.
Now, using \eqref{eq:ML_inequality} we have
$$ \left|\sum_{i=1}^r \tilde{\beta}_r(t,\varepsilon/n_i)w_i n_i^{-r} \right| \le C\|f\|_\infty \left(\int_{\mathbb{R}^d} \|\nabla^{2r} p_1(\xi) \| d\xi \right) \sum_{i=1}^r |w_i| n_i^{-r} .$$
Since $u_k \to u_\infty = \prod_{\ell \ge 1} (1-2^{-\ell})^{-1} < \infty$, the weights satisfy
$$ |w_i| \le u_\infty^2 2^{-(r-i)(r-i+1)/2}, \quad i=1,\ldots,r, $$
so that
\begin{equation}
\label{eq:wi_ni_b}
\sum_{i=1}^r \frac{|w_i|}{n_i^r} \le u_\infty^2 \sum_{i=1}^r 2^{-(r-i)(r-i+1)/2} \le u_\infty^2 \sum_{i=1}^r 2^{(r-i)/2} = u_\infty^2 \sum_{i=0}^{r-1}2^{-i/2} \le C.
\end{equation}
As a consequence and since $\sum_{i=1}^r w_i =1$, we may write from \eqref{eq:ML_expansion}
\begin{equation}
\label{eq:ML:2}
\left| \mathbb{E}f(Z_1) - \sum_{i=1}^r w_i \mathbb{E}f_{\varepsilon/n_i}(Z_1)\right| \le C\|f\|_\infty \varepsilon^r \int_{\mathbb{R}^d} \|\nabla^{2r} p_1(\xi) \| d\xi .
\end{equation}
The same way, we obtain
$$ \left|\mathbb{E}f(Z_2) - \sum_{i=1}^r w_i \mathbb{E}f_{\varepsilon/n_i}(Z_2) \right| \le C\|f\|_\infty \varepsilon^r \int_{\mathbb{R}^d} \|\nabla^{2r} p_2(\xi) \| d\xi.$$

On the other side, using \eqref{eq:f_epsilon_lip} we have
\begin{equation}
\label{eq:D3:thm2}
\left|\sum_{i=1}^r w_i \mathbb{E}f_{\varepsilon/n_i}(Z_1) - \sum_{i=1}^r w_i \mathbb{E}f_{\varepsilon/n_i}(Z_2) \right| \le \frac{C\|f\|_\infty}{\sqrt{\varepsilon}} \mathcal{W}_1(Z_1,Z_2) \left(\sum_{i=1}^r |w_i| 2^{(i-1)/2} \right). 
\end{equation}
Moreover, for every $i=1,\ldots,r$,
$$ |w_i| 2^{(i-1)/2} \le u_\infty^2 2^{-(r-i)(r-i+1)/2 + (i-1)/2} $$
and then
\begin{equation}
\label{eq:wi_2_i_b}
\sum_{i=1}^r |w_i| 2^{(i-1)/2} \le u_\infty^2 \sum_{i=1}^r 2^{(i-1)/2} \le u_\infty^2 2^r .
\end{equation}
Thus considering \eqref{eq:TI_Z1_Z2}, we obtain for every $\varepsilon > 0$,
$$ \dtv(Z_1,Z_2) \le C \varepsilon^r \int_{\mathbb{R}^d} \left( \| \nabla^{2r} p_1(\xi) \| + \| \nabla^{2r} p_2(\xi) \| \right) d\xi + C \varepsilon^{-1/2} \mathcal{W}_1(Z_1,Z_2). $$
Optimizing in $\varepsilon$ gives
$$ \varepsilon_\star = \left( \mathcal{W}_1(Z_1,Z_2)/ (2r \int_{\mathbb{R}^d} \left( \| \nabla^{2r} p_1(\xi) \| + \| \nabla^{2r} p_2(\xi) \| \right) d\xi) \right)^{2/(2r+1)} $$
and then
$$ \dtv(Z_1,Z_2) \le C_{d,r} \mathcal{W}_1(Z_1,Z_2)^{2r/(2r+1)} \left( \int_{\mathbb{R}^d} \left( \| \nabla^{2r} p_1(\xi) \| + \| \nabla^{2r} p_2(\xi) \| \right) d\xi \right)^{1/(2r+1)}. $$
\end{proof}

We now prove Theorem \ref{thm:main:2}.
\begin{proof}
Using Lemma \ref{lemma:cut_drift}, we have
\begin{align}
\label{eq:dTV_D_6:2}
\dtv(X_t^x, Y_t^x) & \le Ce^{C|x|^2}t^{1/2} + \dtv(Z^x_t,V^x_t)
\end{align}
We now apply Theorem \ref{thm:dtv_bound_L1_density} with the random vectors $Z_1=Z^x_t$ and $Z_2=V^x_t$. Assuming that $\sigma_1$ is $C^{2r}_b$ and using Theorem \ref{thm:friedman}, $\nabla^k_y p_{_Z}$ exists for $k=0,1,\ldots,2r$ and
$$ \forall k=0,1,\ldots,2r, \ \forall t \in (0,T], \ \forall x,y \in \mathbb{R}^d, \ \| \nabla^k_y p_{_Z}(0,t,x,y) \| \le \frac{C}{t^{(d+k)/2}} e^{-c|y-x|^2/t} .$$
Then we have
$$ \int_{\mathbb{R}^d} \nabla^{2r}_y p_{_Z}(0,t,x,\xi) d\xi \le Ct^{-r} \int_{\mathbb{R}^d} \frac{1}{t^{d/2}} e^{-c|x-\xi+y|^2/t} d\xi \le Ct^{-r}. $$
The same way we have
$$ \int_{\mathbb{R}^d} \nabla^{2r}_y p_{_V}(0,t,x,\xi) d\xi \le Ct^{-r} .$$
Applying Theorem \ref{thm:dtv_bound_L1_density} with Lemma \ref{lemma:strong_error} yields
$$ \dtv(Z^x_t, V^x_t) \le C(\sqrt{t}+\Delta\sigma(x))^{2r/(2r+1)} .$$
%
\end{proof}

\subsection{Proof of Theorem \ref{thm:main:2:menozzi}}

For the proof of Theorem \ref{thm:main:2:menozzi}, we do not use Lemma \ref{lemma:cut_drift}; instead we directly apply Theorem \ref{thm:dtv_bound_L1_density}.
Using Theorem \ref{thm:menozzi}, $\nabla^k_y p_{_X}$ and $\nabla^k_y p_{_Y}$ exist for $k=0,1,\ldots,2r$ and satisfy the same bounds as previously. Then using Theorem \ref{thm:dtv_bound_L1_density} with Lemma \ref{lemma:strong_error} we obtain
$$ \dtv(X^x_t, Y^x_t) \le C(\sqrt{t}(1+\Delta b(x)) + \Delta \sigma(x) + t|b(x)|)^{2r/(2r+1)} .$$

\subsection{Proof of Theorem \ref{thm:main:3}}

\begin{proof}
We use Lemma \ref{lemma:cut_drift} again and rework the bound on $\dtv(Z^x_t, V^x_t)$ by paying attention to the dependency of the constants in $r$ in the proof of Theorem \ref{thm:dtv_bound_L1_density} with $Z_1:=Z^x_t$ and $Z_2:=V^x_t$.
Since $\sigma_1 \in \mathcal{C}^{2r}_{b}$ for every $r \in \mathbb{N}$, we write \eqref{eq:D_2_bound:2} for any $r \in \mathbb{N}$ and we have
\begin{align*}
& |\tilde{\beta}_r(t,\varepsilon)| \le \widetilde{C}_{2r} \|f\|_\infty t^{-r} \frac{\mathbb{E}[|\zeta|^{2r}]}{(2r)!} , \quad \widetilde{C}_{2r} := C_{2r}c_{2r}^{-d/2},
\end{align*}
where $C_{2r}$ and $c_{2r}$ are defined in \eqref{eq:assumption_C_tilde} and where
\begin{align*}
& \mathbb{E}[|\zeta|^{2r}] = \frac{2^r \Gamma(d/2 + r)}{\Gamma(d/2)} = \prod_{i=0}^{r-1} (d+2i) .
\end{align*}
Using \eqref{eq:wi_ni_b} we get
$$ \left|\sum_{i=1}^r \tilde{\beta}_r(t,\varepsilon/n_i)w_i n_i^{-r} \right| \le C \widetilde{C}_{2r} \|f\|_\infty t^{-r} \frac{\prod_{i=0}^{r-1} (d+2i)}{(2r)!} $$
and we obtain as in \eqref{eq:ML:2}:
\begin{align*}
& \left| \mathbb{E}f(Z^x_t) - \sum_{i=1}^r w_i \mathbb{E}f_{\varepsilon/n_i}(Z^x_t)\right| \le \frac{1}{2} \kappa_1 \|f\|_\infty \varepsilon^r t^{-r}, \quad \kappa_1 := C\widetilde{C}_{2r}\frac{\prod_{i=0}^{r-1} (d+2i)}{(2r)!} \\
& \left|\mathbb{E}f(V^x_t) - \sum_{i=1}^r w_i \mathbb{E}f_{\varepsilon/n_i}(V^x_t) \right| \le \frac{1}{2} \kappa_1 \|f\|_\infty \varepsilon^r t^{-r}.
\end{align*}
On the other hand, considering \eqref{eq:D3:thm2} and \eqref{eq:wi_2_i_b} with Lemma \ref{lemma:strong_error} with $\Delta \sigma(x) = 0$ we have
\begin{equation}
\left|\sum_{i=1}^r w_i \mathbb{E}f_{\varepsilon/n_i}(V^x_t) - \sum_{i=1}^r w_i \mathbb{E}f_{\varepsilon/n_i}(Z^x_t) \right| \le \kappa_2 \frac{\|f\|_\infty}{\sqrt{\varepsilon}} t, \quad \kappa_2 := C2^r .
\end{equation}
We now minimize $\kappa_1 \varepsilon^r t^{-r} + \kappa_2 \varepsilon^{-1/2}t$ in $\varepsilon$, giving
$$ \varepsilon_\star = \frac{t^{(2r+2)/(2r+1)}}{(2r\kappa_1)^{2/(2r+1)}} \kappa_2^{2/(2r+1)} $$
and then
\begin{align*}
\kappa_1 \varepsilon_\star^r t^{-r} + \kappa_2 \varepsilon_\star^{-1/2}t \le C \kappa_2^{2r/(2r+1)} \kappa_1^{1/(2r+1)} t^{r/(2r+1)}
\end{align*}
with as $r \to \infty$:
\begin{align*}
\kappa_2^{2r/(2r+1)} \kappa_1^{1/(2r+1)} \sim \widetilde{C}_{2r}^{1/(2r+1)} \left( \prod_{i=0}^{r-1} (d+2i) \right)^{1/(2r+1)} \frac{1}{(2r)!^{1/(2r+1)}} 2^{2r^2/(2r+1)}
\end{align*}
with
\begin{align*}
& \left( \prod_{i=0}^{r-1} (d+2i) \right)^{\frac{1}{(2r+1)}} = \exp\left(\frac{r}{2r {+} 1} \frac{1}{r} \sum_{i=0}^{r-1} \log(d+2i) \right) \le \exp\left(\frac{r}{2r{+}1} \log(d{+}(r{-}1)) \right) \le \sqrt{d{+}r{-}1}, \\
& \frac{1}{(2r)!^{1/(2r+1)}} \sim \frac{e}{2r}, \quad \limsup_{r \to \infty} \widetilde{C}_{2r}^{1/(2r+1)} < \infty
\end{align*}
where we used Assumption \eqref{eq:assumption_C_tilde}, so that
$$ \kappa_2^{2r/(2r+1)} \kappa_1^{1/(2r+1)} \le C \sqrt{d+r-1} \frac{e}{2r} 2^r .$$
Then we have $\dtv(Z^x_t, V^x_t) \le C2^r r^{-1/2} t^{r/(2r+1)}$ and we choose $r(t) = \lfloor \log^{1/2}(1/t) \rfloor$ so that as $t \to 0$,
$$ \dtv(Z^x_t, V^x_t) \le C t^{1/2} \exp\left(C\sqrt{\log(1/t)}\right) .$$

\end{proof}

\section{Counterexample}

In this section we give a counter-example showing that we cannot achieve a bound better than $t^{1/2}$ in general. More specifically, we show that we cannot achieve a bound better than $t^{1/2}$ for the total variation between an SDE and its Euler-Maruyama-scheme in general.
For $x >0$ and $\sigma > 0$, let us consider the one-dimensional process
\begin{equation}
\label{eq:counter_example_def}
Y^x_t = xe^{\sigma W_t},
\end{equation}
where $W$ is a standard Brownian motion. The process $Y$ is solution of the SDE $dY^x_t = (\sigma^2/2) Y^x_t dt + \sigma Y^x_t dW_t$ and its associated Euler-Maruyama schemes reads
\begin{equation}
\label{eq:geometric_BM}
\bar{Y}^x_t = x + (\sigma^2/2) x t + \sigma x W_t \sim \mathcal{N}\left(x(1+t\sigma^2/2),\sigma^2x^2t\right).
\end{equation}

\begin{proposition}
Let $Y$ be the process defined in \eqref{eq:counter_example_def}. Then for small enough $t$ we have
\begin{equation}
\dtv(Y^x_t,\bar{Y}^x_t) \ge C_x t^{1/2}.
\end{equation}
\end{proposition}
\begin{proof}
We have
\begin{equation}
p_{_Y}(t,x,y) = \frac{1}{\sqrt{2\pi\sigma^2 t}} \frac{\exp\left(-\frac{1}{2\sigma^2t}\log^2(y/x)\right)}{y} \mathds{1}_{y \ge 0}
\end{equation}
so that
\begin{align*}
& \dtv(Y^x_t, \bar{Y}^x_t) = \frac{1}{\sqrt{2\pi\sigma^2 t}} \int_{\mathbb{R}} \left|\exp\left(-\frac{\log^2(y/x)}{2\sigma^2t}\right) y^{-1} \mathds{1}_{y \ge 0} - \exp\left(-\frac{(y-x-xt\sigma^2/2)^2}{2\sigma^2 x^2 t}\right)x^{-1} \right|dy \\
& \quad \quad \ge \frac{1}{\sqrt{2\pi \sigma^2}} \int_{-x/\sqrt{t}}^\infty \left|\frac{1}{x+\sqrt{t}y} \exp\left(-\frac{\log^2(1+\sqrt{t}y/x)}{2\sigma^2 t}\right) - \frac{1}{x}\exp\left(-\frac{(y-x\sqrt{t}\sigma^2/2)^2}{2\sigma^2 x^2}\right) \right| dy.
\end{align*}
But we have as $(t,y) \to 0$:
\begin{align*}
& \frac{1}{1+\sqrt{t}y/x} \exp\left(-\frac{\log^2(1+\sqrt{t}y/x)}{2\sigma^2 t}\right) - \exp\left(-\frac{(y-x\sqrt{t}\sigma^2/2)^2}{2\sigma^2 x^2}\right) \\
& \quad = (1- \sqrt{t}y/x + O(ty^2)) \exp\left(-\frac{1}{2\sigma^2 t} (\frac{ty^2}{x^2} - \frac{t^{3/2}y^3}{x^3} + O(t^2 y^4))\right) - \exp\left(- \frac{y^2}{2\sigma^2x^2} - \frac{t\sigma^2}{8} + \frac{\sqrt{t}y}{2\sigma^2 x} \right) \\
& \quad = e^{-\frac{y^2}{2 \sigma^2 x^2}} \left[ (1- \sqrt{t}y/x + O(ty^2)) \left(1+\frac{\sqrt{t}y^3}{2\sigma^2 x^3} + O(ty^4)\right) - \left(1+\frac{\sqrt{t}y}{2\sigma^2 x}-\frac{t\sigma^2}{8} + O(t^2) + O(ty^2) \right) \right] \\
& \quad = e^{-\frac{y^2}{2 \sigma^2 x^2}} \left[ -\frac{\sqrt{t}y}{x} - \frac{\sqrt{t}y}{2\sigma^2 x} + \frac{\sqrt{t}y^3}{2\sigma^2 x^3} + \frac{t\sigma^2}{8} + O(ty^2) + O(t^2) \right].
\end{align*}
Thus there exists $\epsilon >0$ and $t_0$ such that for every $t \le t_0$:
\begin{align*}
\dtv(Y^x_t, \bar{Y}^x_t) \ge \frac{1}{\sqrt{2\pi\sigma^2 x^2}} e^{-\frac{\varepsilon^2}{2\sigma^2 x^2}} \frac{\sqrt{t}}{2} \int_{-\varepsilon}^{\varepsilon} \left| -\frac{y}{x} - \frac{y}{2\sigma^2 x} + \frac{y^3}{2\sigma^2 x^3} \right| dy,
\end{align*}
so that $\dtv(Y^x_t, \bar{Y}^x_t)$ is of order $t^{1/2}$ as $t \to 0$.
\end{proof}

However, the process $Y$ does not satisfy the assumptions of Theorem \ref{thm:main:2} as its noise coefficient is not elliptic neither bounded on $(0,\infty)$. We then prove the following result.
\begin{proposition}
There exists a diffusion process $X$ on $\mathbb{R}$ with $\mathcal{C}^1$ and Lipschitz continuous drift, with $\mathcal{C}^\infty_b$ and elliptic diffusion coefficient $\sigma$ and there exists $T>0$ and $\varepsilon \in (0,1)$ such that
$$ \forall t \in [0,T], \ \forall x \in (\varepsilon, \varepsilon^{-1}), \quad \dtv(X^x_t, \bar{X}^x_t) \ge C_x t^{1/2} $$
where $\bar{X}$ is the Euler-Maruyama scheme of $X$ and where the positive constant $C_x$ depends on $x$.
\end{proposition}
\begin{proof}
We construct from the geometric Brownian motion $Y$ defined in \eqref{eq:counter_example_def}, a process $X$ with elliptic and bounded drift and such that $\dtv(X^x_t, \bar{X}^x_t) \ge C_x t^{1/2}$. For $\varepsilon \in (0,1/2) $, let us consider $\psi : \mathbb{R} \to \mathbb{R}^+$ a $\mathcal{C}^\infty_b$ approximation of
$$ \widetilde{\psi} : x \in \mathbb{R} \longmapsto \left\lbrace \begin{array}{ll}
x & \text{ if } x \in [\varepsilon, \varepsilon^{-1}], \\
\varepsilon & \text{ if } x \le \varepsilon \\
\varepsilon^{-1} & \text{ if } x \in [\varepsilon^{-1}, \infty)
\end{array} \right. $$
such that $\psi = \widetilde{\psi}$ on $[2\varepsilon, \varepsilon^{-1}/2] \cup (-\infty,\varepsilon/2] \cup [2\varepsilon^{-1}, \infty)$.
Then we define the process with elliptic and bounded noise coefficient
$$ dX^x_t = -\frac{\sigma^2}{2} X^x_t dt + \sigma \psi(X^x_t) dW_t. $$
Then for $x \in (2\varepsilon, \varepsilon^{-1}/2)$ we have $\bar{X}^x_t = \bar{Y}^x_t$ and
\begin{align*}
\mathbb{P}(Y^x_t \ne X^x_t) & \le \mathbb{P}\left( \textstyle \sup_{s \in [0,t]} Y^x_s \ge \varepsilon^{-1}/2 \right) + \mathbb{P}\left( \textstyle \inf_{s \in [0,t]} Y^x_s \le 2\varepsilon \right).
\end{align*}
With a proof similar to the proof of Lemma \ref{lemma:dTV_X_Y}, we show that
$$ \mathbb{P}\left( \textstyle \sup_{s \in [0,t]} Y^x_s \ge \varepsilon^{-1}/2 \right) \le C_{x,\varepsilon} t .$$
Moreover, we remark that $(Y^x)^{-1} \sim x^{-2} Y^x$ in law so
$$ \mathbb{P}\left( \textstyle \inf_{s \in [0,t]} Y^x_s \le 2\varepsilon \right) = \mathbb{P}\left( \textstyle \sup_{s \in [0,t]} (Y^x_s)^{-1} \ge \varepsilon^{-1}/2 \right) = \mathbb{P}\left( \textstyle \sup_{s \in [0,t]} Y^x_s \ge x^2 \varepsilon^{-1}/2 \right) \le C_{x,\varepsilon}t .$$
Then we obtain
\begin{align*}
\dtv(X^x_t,\bar{X}^x_t) \ge \dtv(Y^x_t,\bar{Y}^x_t) - \dtv(X^x_t,Y^x_t) \ge C_x \sqrt{t} .
\end{align*}

\end{proof}

\begin{remark}
We could also consider the process $X$ with "cut" bounded drift $\tilde{b}$ and get the same bounds, proving then that we cannot achieve better bounds in general than the ones established in Theorem \ref{thm:main:3} even if we assume that $b$ is bounded.
\end{remark}

\appendix

\section{Appendix}

\begin{lemma}[\cite{friedman}, Chapter 9, Lemma 7]
\label{lemma:friedman_Lemma_7}
For $a>0$, $0 < u < t \le T$, $x \in \mathbb{R}^d$, $\xi \in \mathbb{R}^d$, let
\begin{align*}
& I_a := \int_{\mathbb{R}^d} \frac{1}{(u(t-u))^{d/2}} \exp\left(-a \left(\frac{|x-y|^2}{t-u} + \frac{|y-\xi|^2}{u} \right) \right) dy.
\end{align*}
Then there exists a constant $C>0$ depending only on $d$ and $T$ such that for every $ 0 < \varepsilon < 1$,
$$ I_a \le \frac{C}{(\varepsilon a t)^{d/2}} \exp \left(-a(1-\varepsilon) \frac{|x-\xi|^2}{t} \right) .$$
\end{lemma}

Let us recall \cite[Lemma 3.4(a)]{pages2020}, with an immediate adaptation to the non-homogeneous case.
\begin{lemma}
\label{lemma:strong_error_x}
Let $Z$ be solution to the generic SDE:
$$ Z^x_0 = x \in \mathbb{R}^d, \quad dZ^x_t = b(t,Z^x_t)dt + \sigma(t,Z^x_t)dW_t, \ t \in [0,T],$$
where $b$ and $\sigma$ are Lipschitz continuous in $(t,x)$ and where $\sigma$ is bounded.
Then for $p \ge 1$,
$$ \forall t \in [0,T], \ \forall x \in \mathbb{R}^d, \quad \|Z^x_t - x \|_p \le C(p,T,\lip{b},\lip{\sigma},\|\sigma\|_\infty) \left( t|b(0,x)| + t^{1/2} \right).$$
\end{lemma}

\begin{lemma}
\label{lemma:strong_error}
Let $X$ and $Y$ be the solution to the two general SDEs \eqref{eq:SDE_def:1} and \eqref{eq:SDE_def:2}. Assume that $b$ and $\sigma$ are Lipschitz continuous in $(t,x)$ and that $\sigma$ is bounded. Then for every $p \ge 1$,
\begin{align}
& \forall t \in [0,T], \ \forall x \in \mathbb{R}^d, \ \|X^x_t - Y^x_t \|_p \le C \left(t(1 +\Delta b(x)) + t^{3/2}(|b_1|+|b_2|)(0,x) + \Delta \sigma(x)t^{1/2} \right)
\end{align}
\end{lemma}

\begin{proof}
We first deal with the case $p \ge 2$. We have
\begin{align*}
& \|X^x_t - Y^x_t\|_p \le \left\| \int_0^t (b_1(s,X^x_s) - b_2(s,Y^x_s))ds \right\|_p + \left\| \int_0^t (\sigma_1(s,X^x_s) - \sigma_2(s,Y^x_s))dW_s \right\|_p \\
& \quad \le \left\| \int_0^t (b_1(s,X^x_s)-b_1(0,x))ds \right\|_p + t\Delta b(x) + \left\| \int_0^t (b_2(s,Y^x_s)-b_2(0,x))ds \right\|_p \\
& \quad \quad + \left\| \int_0^t (\sigma_1(s,X^x_s)-\sigma_1(0,x))dW_s \right\|_p + \left\| \int_0^t \Delta \sigma(x)dW_s \right\|_p + \left\| \int_0^t (\sigma_2(s,Y^x_s)-\sigma_2(0,x))dW_s \right\|_p
\end{align*}
But using the Burkholder-Davis-Gundy and the generalized Minkowski inequalities, we have
\begin{align*}
& \left\| \int_0^t (\sigma_1(s,X^x_s) - \sigma_1(0,x))dW_s \right\|_p \le C^{\text{BDG}}_p \lip{\sigma_1} \left\| \int_0^t |(s,X^x_s) - (0,x)|^2 ds \right\|_{p/2}^{1/2} \\
& \quad \le C^{\text{BDG}}_p \lip{\sigma_1} \left(\int_0^t \| (s,X^x_s) - (0,x) \|_p^2 ds \right)^{1/2} \le C(t + t^{3/2}|b_1(0,x)|),
\end{align*}
where $C^{\text{BDG}}_p$ is a constant which only depends on $p$ and where we used Lemma \ref{lemma:strong_error_x}. So that
\begin{align*}
& \|X^x_t - Y^x_t\|_p \le \lip{b_1} \int_0^t \|(s,X^x_s) - (0,x) \|_p ds + \lip{b_2}\int_0^t \|(s,Y^x_s) - (0,x)\|_p + t\Delta b(x) \\
& \quad \quad + C(t + t^{3/2}(|b_1|+|b_2|)(0,x)) + \Delta \sigma(x) \sqrt{t}\|W_1\|_p \\
& \quad \le C \left( t(\Delta b(x) +1) + t^{3/2}(|b_1|+|b_2|)(0,x) + \Delta \sigma(x)\sqrt{t} \right)
\end{align*}
which completes the proof for $p \ge 2$. For $p \in [1,2)$, the inequality is still true remarking that $\| \cdot \|_p \le \| \cdot \|_2$.
\end{proof}


\end{document}